\documentclass[oneside, 10pt]{amsart}

\usepackage{amsmath, amsfonts, amssymb, amsthm, graphics, times}

\newtheorem{theorem}{Theorem}[section]

\theoremstyle{definition}

\numberwithin{equation}{section}

\begin{document}

\title{Riesz's Theorem  for Lumer's Hardy Spaces}

\author{Marijan Markovi\'{c}}
\address{
Faculty of  Sciences and Mathematics\endgraf
University of Montenegro\endgraf
D\v{z}ord\v{z}a Va\v{s}ingtona bb\endgraf
81000 Podgorica\endgraf
Montenegro\endgraf
}

\email{marijanmmarkovic@gmail.com}

\begin{abstract}
In this      note we obtain a version of the well-known Riesz's theorem on conjugate harmonic functions for
Lumer's Hardy spaces $(Lh)^2(\Omega)$ on   arbitrary  domains   $\Omega$:         If a real-valued harmonic
function $U\in (Lh)^2(\Omega)$   has a harmonic conjugate $V$ on $\Omega$     (i.e., a real-valued harmonic
function such  that $U+ iV$ is analytic on $\Omega$), then $U+iV$ also belongs to $(Lh)^2(\Omega)$, and for
the normalized conjugate  we have the norm estimate                   $\|U+iV\|_{(Lh)^2(\Omega)}\le\sqrt{2}
\|U\|_{(Lh)^2(\Omega)}$, with the best possible constant.
\end{abstract}

\maketitle

\section{On Lumer's Hardy spaces}
Let $\mathbf{U}=\{z\in \mathbf{C}:|z|<1\}$ be the unit disk  and let $\mathbf{T}=\{z\in \mathbf{C}:|z|=1\}$
be the  unit circle.   For a function  $f$ on $\mathbf{U}$ and $r\in (0,1)$ we denote by $f_r$ the function
$f_r (\zeta) =  f(r\zeta)$, $\zeta\in  \overline{\mathbf{U}}$.

The harmonic Hardy space $h^p$, for   $p\in (1,\infty)$, consists of all harmonic complex-valued  functions
$U$   on $\mathbf{U}$ for  which the   integral mean
\begin{equation*}
M_p(U,r)  = \left\{ \int_\mathbf{T} |U_r(\zeta)|^p\frac{|d\zeta|}{2\pi}\right\}^{1/p}
\end{equation*}
remains bounded as $r$ approaches $1$. Since $|U|^p$ is  subharmonic on  $\mathbf{U}$,   the integral  mean
$M_p(U,r)$ is  increasing in $r$. The  norm  on $h^p$ is  given by
\begin{equation*}
\|U\|_p = \lim_{r\rightarrow 1} M_p(U,r).
\end{equation*}
The analytic Hardy space $H^p$ is the subspace of $h^p$ that contains all analytic functions.       For the
theory of                                                         Hardy spaces in the unit disk we refer to
\cite{DUREN.BOOK, FISHER.BOOK, GARNETT.BOOK, RUDIN.BOOK}.

There  are  generalizations of Hardy spaces for other   domains in $\mathbf{C}$. The generalizations     we
consider here                                                          are known   as  Lumer's Hardy spaces
\cite{DUREN.BOOK, FISHER.BOOK, LUMER.PARIS, LUMER-NAIM, RUDIN.TAMS}.          We mention   below some facts
regarding  these   spaces that we    will      need.

The harmonic    Lumer's Hardy  space $(Lh)^p(\Omega)$ contains all harmonic complex-valued   functions  $U$
on a domain $\Omega\subseteq \mathbf{C}$ such that         the subharmonic function  $|U|^p$ has a harmonic
majorant on $\Omega$.     In that case, the  function  $|U|^p$ has the  least harmonic majorant on $\Omega$.
Let it be denoted by  $H_U$.  For $\zeta_0\in \Omega$ one introduces a  norm  on  $(Lh)^p (\Omega)$  in the
following way:
\begin{equation}\label{EQ.LN}
\|U\|_ {p,\zeta_0} = H_U ^{1/p} (\zeta_0).
\end{equation}
The   different  norms on  $(Lh)^p(\Omega)$      that arise by  selecting different elements of the  domain
$\Omega$ are mutually equivalent.  The analytic Lumer's Hardy  space   $(LH)^p(\Omega)$ is the subspace  of
$(Lh)^p(\Omega)$   that  consists of all analytic functions. The two spaces  $(Lh)^p(\mathbf{U})$ and $h^p$
coincide (as do $(LH)^p(\mathbf{U})$ and $H^p$). The norms on these spaces are equal, if we select $\zeta_0
=0$ for the Lumer  case.

Lumer's Hardy  spaces are conformally invariant in the following sense: If $\Phi$ is a conformal mapping of
a domain $\tilde{\Omega}$ onto $\Omega$,  then a  function  $U$ belongs  to    $(Lh)^p(\Omega)$ if and only
if  $\tilde{U}=U\circ\Phi$  belongs to $(Lh)^p (\tilde{\Omega})$. The mapping $\Phi$ induces  an  isometric
isomorphism $U\rightarrow\tilde{U}$ of the space $(Lh)^p(\Omega)$ onto $(Lh)^p (\tilde{\Omega})$, since the
equality for the least harmonic majorants $H_U \circ\Phi = H_{\tilde{U}}$ implies that $\|U\|_{p,\zeta_0} =
\|\tilde{U}\|_{p,\tilde{\zeta_0}}$, where $\tilde{\zeta_0} \in \tilde{\Omega}$ satisfies      $\zeta_0=\Phi
(\tilde{\zeta_0})$.

\section{Riesz's theorem for Lumer's Hardy spaces}
The classical  Riesz theorem on  conjugate harmonic functions  says that for every  $p\in (1,\infty)$ there
exists  a constant  $c_p$ such that
\begin{equation*}
\|U+  i V\|_p\le c_p \|U\|_p,
\end{equation*}
where $U$ is a real-valued   function in $h^p$,        $V$ is a harmonic  conjugate  to $U$ on $\mathbf{U}$,
normalized such that $V(0)=0$.  See, for  instance,    \cite[Theorem 17.26]{RUDIN.BOOK}.   Verbitsky proved
\cite{VERBITSKY.ISSLED} that the best possible  constant in the Riesz inequality  is
\begin{equation}\label{EQ.VC}
c_p =
\begin{cases}
\mathrm {sec} \frac \pi{2p}, & \mbox{if}\  1<p\le 2; \\
\mathrm {csc} \frac \pi{2p}, & \mbox{if}\  2\le p<\infty.
\end{cases}
\end{equation}
Note that,  in particular,      we have  $c_2 = \sqrt{2}$.

Our  aim in this section  is to prove the Riesz       theorem for real-valued        harmonic functions  in
the Lumer's Hardy space $(Lh)^2(\Omega)$ for  which  there  exists a conjugate. We find  that the  constant
$\sqrt{2}$ is  valid for all domains $\Omega$. This is the content  of the  following theorem.

\begin{theorem}\label{TH.MAIN}
Let  $\Omega\subseteq  \mathbf{C}$ be a domain and $\zeta_0\in \Omega$. Assume  that for real-valued $U \in
(Lh)^2(\Omega)$  there exists a harmonic conjugate of $U$  on the  domain $\Omega$, denoted by $V$, and let
it be normalized such that $V(\zeta_0)=0$. Then  we   have the  Riesz  inequality
\begin{equation}\label{EQ.R}
\|U+iV\|_{2,\zeta_0}\le \sqrt{2}\|U\|_{2,\zeta_0},
\end{equation}
with the best  possible  constant.
\end{theorem}

\begin{proof}
We will  use the following elementary equality, which is
easy to check:
\begin{equation}\label{EQ.Z}
|z|^2  = 2 (\Re z)^2 - \Re z^2,\quad z\in\mathbf{C}.
\end{equation}
Indeed,  since $2\Re z = z+\overline{z} $, we have
\begin{equation*}
4 (\Re z)^2      =         (z+\overline{z})^2  = z^2 + \overline{z}^2 +2z\overline {z} = 2\Re z^2 + 2|z|^2;
\end{equation*}
the equality mentioned above then follows.

Let   the analytic function    $U + iV$ be denoted  by  $F$, and let $H_U$ be  the  least harmonic majorant
of  the subharmonic function  $|U |^2$ on  $\Omega$. By applying equation  \eqref{EQ.Z} for $z  = F(\zeta)$,
$\zeta\in \Omega$,                     we obtain
\begin{equation*}\begin{split}
|F(\zeta)|^2&=2 (\Re F(\zeta))^2 - \Re F^2(\zeta)= 2  |U(\zeta) | ^2 - \Re  F^2(\zeta)
 \\& \le 2H_U (\zeta) -   \Re  F^2(\zeta),
\end{split}\end{equation*}
which proves that $2H_U  -\Re  F ^2 $                   is a harmonic   majorant  of  $| F |^2$ on $\Omega$.
It follows that $F\in (LH)^2(\Omega)$. Moreover, if $H_F$ is the least harmonic majorant     of  $|F|^2$ on
$\Omega$, we have
\begin{equation*}
H_F (\zeta)   \le 2 H_U(\zeta) - \Re F^2(\zeta).
\end{equation*}
Since $F(\zeta_0) = U(\zeta_0)$ is a real  number,  we  obtain
\begin{equation*}\begin{split}
\|F\|^2_{2,\zeta_0} & = H_F(\zeta_0)  \le 2 H_U(\zeta_0) - \Re F^2(\zeta_0) =  2 H_U(\zeta_0) -U^2(\zeta_0)
\\&\le 2 H_U(\zeta_0)   =   2\|U\|^2_{2,\zeta_0}.
\end{split}\end{equation*}
Finally, we conclude that
\begin{equation*}
\|F\|_{2,\zeta_0} \le \sqrt {2}\|U\|_{2,\zeta_0},
\end{equation*}
which is what we wanted to prove.

It is  not hard to prove that $\sqrt{2}$ is a sharp constant in the Riesz inequality \eqref{EQ.R}.   Indeed,
consider  the unit disk  $\mathbf{U}$ as the domain $\Omega$. If we again use equation \eqref{EQ.Z} for  $z
=  F(\zeta)$,   we have
\begin{equation*}
|F (\zeta)|^2  = 2 U^2 (\zeta)  - \Re F^2(\zeta).
\end{equation*}
Since $\Re F^2 $ is a harmonic function on $\mathbf{U}$, by applying the equality obtained    above and the
mean-value property  for harmonic functions,  it follows that
\begin{equation*}\begin{split}
M_2^2 ( F, r)  & =  \int_ \mathbf{T} |F_r( \zeta)|^2 \frac{|d\zeta|}{2\pi}
=2\int_\mathbf{T}U_r^2(\zeta)\frac{|d\zeta|}{2\pi}-\int _\mathbf{T}\Re{F^2_r(\zeta)}  \frac{|d\zeta|}{2\pi}
\\&= 2 M_2^2 (U, r) -  \Re F^2(0)= 2 M_2^2 (U, r) -   U^2(0).
\end{split}\end{equation*}
If we  now   let  $r\rightarrow 1$,  we  obtain
\begin{equation*}\begin{split}
 \|F\|_{2 } = \sqrt{2} \|U \|_{2 }
\end{split}\end{equation*}
provided that   $U(0)=0$.
\end{proof}

Note that the constant $\sqrt{2}$ in the Riesz  inequality  \eqref{EQ.R}    does not depend  on $\zeta_0\in
\Omega$,  although the  norm of a function  in the  Lumer's  Hardy space $(Lh)^2(\Omega)$ does. If $\Omega$
is a  simply connected  domain with at least two boundary points, this  is  expected, since the       group
of all conformal automorphisms  of the domain         $\Omega$ acts transitively on $\Omega$, i.e., for any
$\tilde{\zeta}_0\in\Omega$ there exists a  conformal automorphism   $\Phi$ of $\Omega$            such that
$\Phi(\tilde{\zeta}_0)  = \zeta_0$.    As we have already  said, the  mapping $\Phi$ induces  an  isometric
isomorphism of $(Lh)^2(\Omega)$ onto itself.  However, for multi-connected domains     it is  not true,  in
general, that the group  of all   conformal  automorphisms   acts  transitively on a domain.

\section{Remarks on the higher-dimensional setting  and a conjecture}
Lumer's Hardy spaces $(Lh)^p (\Omega)$ and $(LH)^p (\Omega)$ on   domains $\Omega$ in $\mathbf{C}^n$    are
defined in a similar way  as in the one-dimensional case \cite{LUMER.PARIS}.        However, instead of the
harmonic     majorant  we have to use a pluriharmonic majorant, i.e., a function that is   locally the real
part of an  analytic function on  $\Omega$. Therefore, the  Lumer's Hardy space  $(Lh)^p (\Omega)$ contains
all pluriharmonic functions $U$ on $\Omega$  such that $ |U|^p$ has a pluriharmonic majorant    on $\Omega$.
The analytic Lumer's  Hardy space $(LH)^p (\Omega)$ is the subspace of  $(Lh)^p (\Omega)$     that consists
of all analytic functions.  The norm on $(Lh)^p (\Omega)$  may be introduced with respect to  any  $\zeta_0
\in\Omega$ using  the least  pluriharmonic majorant as  in  the ordinary case \eqref{EQ.LN}.

The proof of Riesz's theorem given in the preceding section may be  adapted directly for     Lumer's  Hardy
spaces on domains in  $\mathbf{C}^n$. Therefore, the Riesz inequality \eqref{EQ.R} remains  valid, with the
same constant  $\sqrt{2}$,  in this setting.

We conjecture that for    every   $p\in (1,\infty)$ and  every domain $\Omega\subseteq\mathbf{C}^n$   there
holds  the version of Riesz's   theorem:             Let $F$ be an analytic function on  $\Omega$ such that
$\Im F (\zeta_0) = 0$; if $\Re  F \in (Lh)^p(\Omega)$, then $F\in (LH)^p(\Omega)$   and there is the  Riesz
inequality
\begin{equation*}
\|F\|_{p,\zeta_0}\le c_p \|\Re F\|_{p,\zeta_0},
\end{equation*}
where $c_p$ is   the Verbitsky constant \eqref{EQ.VC}.

In seventies,    Stout \cite[Theorem  IV.1]{STOUT.AJM}  proved   Riesz's  theorem  for Lumer's Hardy spaces
$(LH)^p(\Omega)$ on       $\mathcal {C}^2$-smooth domains $\Omega\subseteq \mathbf{C}^n$ (without a precise
constant in the Riesz inequality).    In this case there exists an integral representation  of  the Lumer's
norm of an    analytic function  that   is  used in order  to obtain the result.

\section{Acknowledgment}
The author wishes  to thank to the referees  and the editors of the \textsc{Monthly} for  help  in improving  the work.


\begin{thebibliography}{10}

\bibitem{DUREN.BOOK}
Duren,  P. (1970).
\textit{Theory of $H^p$ spaces}.
New York and London: Academic Press.

\bibitem{FISHER.BOOK}
Fisher,  S. D.  (1983).
\textit{Function Theory on Planar Domains}.
New York: Wiley Interscience.

\bibitem{GARNETT.BOOK}
Garnett, J. B. (2007).
\textit{Bounded Analytic Functions}.
New York: Springer.

\bibitem{LUMER.PARIS}
Lumer,  G. (1971).
{Espaces de Hardy en plusieurs variables complexes}.
\textit{C. R. Acad. Sci. Paris S\'{e}r. A-B.}  273: 151--154.

\bibitem{LUMER-NAIM}
Lumer-Na\"{\i}m,  L. (1967).
{$\mathcal{H}^p$ spaces  of  harmonic functions}.
\textit{Ann. Inst. Fourier (Grenoble).}  17(2): 425--469.

\bibitem{RUDIN.TAMS}
Rudin,  W. (1955).
{Analytic functions of class $H_p$}.
\textit{Trans. Amer. Math. Soc.}  78(1): 46--66.

\bibitem{RUDIN.BOOK}
Rudin,  W. (1987).
\textit{Real and Complex Analysis}.
New York:  McGraw-Hill.

\bibitem{STOUT.AJM}
Stout,  E. L. (1976).
{$H^p$-functions on strictly pseudoconvex domains}.
\textit{Amer. J. Math.} 98(3): 821--852.

\bibitem{VERBITSKY.ISSLED}
Verbitsky,  I. E. (1984).
{Estimate of the norm of a function in a Hardy Space in terms of the norms of its real and imaginary parts}.
\textit{Amer. Math. Soc. Transl.}  124: 11--15.

\end{thebibliography}
\end{document}